\documentclass[final,12pt]{colt2023} 

\title[Sample Complexity Using Infinite Multiview Models]{Sample Complexity Using Infinite Multiview Models}
\usepackage{times}
\DeclareMathAlphabet{\pazocal}{OMS}{zplm}{m}{n}
\def\rn{\mathbb{R}}

\def\nn{\mathbb{N}}

\def\tf{\tilde{f}}
\def\tp{\tilde{p}}
\def\tw{\tilde{w}}

\def\tL{\tilde{L}}

\def\sD{\pazocal{D}}

\def\sF{\pazocal{F}}

\def\sH{\pazocal{H}}

\def\d2{\sD_2}

\def\ind{\mathbbm{1}}
\def\simiid{\overset{iid}{\sim}}
\def\spn{\operatorname{span}}

\def\cip{\overset{p}{\rightarrow}}

\def\lip{\operatorname{Lip}}
\def\proj{\operatorname{Proj}}
\def\spn{\operatorname{span}}
\def\poly{\operatorname{poly}}
\def\cone{\operatorname{cone}}
\def\supp{\operatorname{supp}}
\def\nlc{\mathfrak{V}}

\usepackage{bbm,csquotes,mathrsfs}

\hypersetup{breaklinks=true}

\coltauthor{%
 \Name{Robert A. Vandermeulen} \Email{vandermeulen@tu-berlin.de}\\
 \addr Technische Universit\"at Berlin: Machine Learning Group\\
 Berlin Institute for the Foundations of Learning and Data
}

\begin{document}

\maketitle

\begin{abstract}%
  Recent works have demonstrated that the convergence rate of a nonparametric density estimator can be greatly improved by using a low-rank estimator when the target density is a convex combination of separable probability densities with Lipschitz continuous marginals, i.e. a multiview model. However, this assumption is very restrictive and it is not clear to what degree these findings can be extended to general pdfs. This work answers this question by introducing a new way of characterizing a pdf's complexity, the \emph{non-negative Lipschitz spectrum} (NL-spectrum), which, unlike smoothness properties, can be used to characterize virtually any pdf. Finite sample bounds are presented that are dependent on the target density's NL-spectrum. From this dimension-independent rates of convergence are derived that characterize when an NL-spectrum allows for a fast rate of convergence.
\end{abstract}


\section{Introduction}
Nonparametric density estimation is a statistical task whose mathematical properties are very well-studied. The universal consistency of popular nonparametric density estimators, like the histogram and kernel density estimator, has been known for some time. In addition, there exist finite sample bounds and rates of convergence for nonparametric density estimators when the target density is known to come from a smooth class of densities \citep{silverman78,gyorfi85,devroye01,tsybakov08,vandermeulen13,dasgupta14,jiang17}. It is well-known that nonparametric density estimation suffers strongly from the curse of dimensionality. In theoretic works this typically manifests as a dimensionality exponent somewhere in rates or bounds (see Theorem 1 in \cite{jiang17}, for example). Recently it has been proven that combining smoothness assumptions with a low-rank/multiview assumption can drastically improve the rate of convergence of a nonparametric density estimator, obviating the curse of dimensionality. In particular, it has been proven that there exist universally consistent nonparametric density estimators that converge at rate $\tilde{O}_n\left(1/\sqrt[3]{n}\right)$ whenever the target density satisfies a \emph{multiview model} assumption \citep{vandermeulen21},
\begin{equation}\label{eqn:mv-finite}
  p\left(x_1,\ldots,x_d\right) = \sum_{i=1}^k w_i \prod_{j=1}^d p_{i,j}\left(x_j\right),
\end{equation}
where $p_{i,j}$ are Lipschitz continuous probability density functions (pdfs) and $w$ lies in the probability simplex. Remarkably this rate is independent of dimension, $d$, the number of components, $k$, or the Lipschitz constants of the component marginal densities. On the same class of densities it was also shown that the standard histogram estimator converges at rate $\omega\left(1/\sqrt[d]{n}\right)$ regardless of choice of rate on bin width \citep{vandermeulen21}, which is a clear instance of the curse of dimensionality. 

While the result above demonstrates the potential benefits of incorporating multiview structure into density estimation, the assumption \eqref{eqn:mv-finite} is quite strong. This work extends the analysis of \cite{vandermeulen21} so that it is applicable to virtually any pdf. To do this, a new characterization of density complexity, the \emph{non-negative Lipschitz spectrum} (NL-spectrum) is introduced. The NL-spectrum is an infinite sum ($k=\infty$) extension of \eqref{eqn:mv-finite}, that characterizes how fast $w_i$ decays and the Lipschitz constants grow. The NL-spectrum is then shown to be applicable to an extremely general class of pdfs: every (Lebesgue) almost everywhere (a.e.) continuous pdf has an NL-spectrum. This enables the extension of the analysis in \cite{vandermeulen21} to a much larger class of pdfs. 
A finite sample bound depending on NL-spectra is then derived for the low-rank histogram estimators introduced in \cite{vandermeulen21}. Finally this bound is used to show rates of convergence of estimators based on NL-spectra rates of growth and decay. In particular, in the infinite component version of \eqref{eqn:mv-finite}, if it is known that $\sum_{i=k+1}^\infty w_i \in O_k\left(k^{-\alpha}\right)$ and the Lipschitz constants of $p_{k,j}$, $L_k$ satisfy, $L_k \in O_k\left(k^\beta\right)$ then there exists a universally consistent estimator that converges at rate $\tilde{O}_n\left(n^{-\alpha/(3\alpha + \beta +1)}\right)$.

\subsection{Related Work}
The first work to investigate nonparametric density estimators for multiview models was \cite{song14}, which proposed a method for factorizing a kernel density estimate to recover a multiview model. While \cite{song14} contained theoretical guarantees for the rate of recovery of the multiview components, it did not demonstrate that the multiview assumption could be leveraged to improve estimator convergence. Instead it was proposed an approach to nonparametric mixture modeling. Other approaches to nonparametric mixture modeling with strong theoretical guarantees include assuming the data is grouped with respect to mixture components \citep{vandermeulen15,vandermeulen19,ritchie21,vandermeulen22} or that the are mixture components satisfy concentration assumptions \citep{dan18,aragam20,aragam21,aragam22}. Like \cite{song14}, these approaches did not produce rates of convergence that beat typical nonparametric rates.

Other investigations into low-rank density estimation have focused on identifiability and recoverability \citep{allman09,kargas19}, while other works have proposed low-rank methods that are empirically shown to improve nonparametric estimation without rate guarantees \citep{song13,novikov21,amiridi22}.
Low-rank approaches to matrix estimation have also been studied extensively. Non-negative matrix or tensor factorization is a task similar to multiview density estimation since it can be used to recover a low-rank probability matrix/tensor \citep{lee99,donoho04,kim07,arora12}. Low-rank matrix methods have been studied extensively in the field of \emph{compressed sensing} which has produced improved matrix estimators with strong theoretical guarantees using the restricted isometry property \citep{recht10} or restricted strong convexity \citep{negahban11,negahban12}. Although a bit different than the methods presented so far, enforcing low non-negative rank for coupling measures in Wasserstein distance estimation has also been investigated. This has been shown to yield improved statistical estimation with computational benefits \citep{scetbon21,scetbon22}. While distance estimation is fairly different from density estimation, this method is noteworthy since it optimizes over a class of low-rank probability measures and has strong theoretical analysis demonstrating improved estimator convergence.  

The works \cite{vandermeulen20,vandermeulen21} are the first to show, via strong theoretic guarantees, that a multiview assumption can be used to improve the rate of convergence of nonparametric density estimators. Those works also show that a non-negative Tucker factorization can also be used to this effect and found that Tucker factorization produced better estimators in practice. Similarly to the proofs in those works, the results in this paper are \emph{not} adaptations of techniques developed for compressed sensing or non-negative matrix factorization.

\section{Results} \label{sec:results}
This section presents and discusses the main results of this paper. Proofs of all results can be found in Section \ref{sec:proofs}.

Before introducing the results, some notation and terminology needs to be introduced. Other notation will be introduced intermittently through this work and a table summarizing notation in this work can be found in Appendix \ref{appx:notation}.
For a pair of sets, $A$ and $B$, $A\times B$ denotes the Cartesian product.
For a pair of real-valued functions, $f:A \to \rn$ and $g:B \to \rn$, their product is defined as $f\times g:A\times B\to \rn$: $\left(a, b\right) \mapsto f(a)g(b)$. Note that if $A=B$ in this case then $f\times g$ is a function on $A\times A$, \emph{not} a function on $A$; $f\cdot g: x\mapsto f(x)g(x)$.  $\nn$ denotes all integers greater than 0. For a set $A$, $\ind_A$ denotes the indicator function on $A$. For $n \in \nn$, $\left[n\right] = \left\{1,2,\ldots,n\right\}$. For sets and functions, the product operator $\prod$ and power operator $\cdot^{\times n}$ will always mean the products, $\times$, defined above.
 The term \emph{almost everywhere} (a.e.) will always refer to the Lebesuge measure. A \emph{pdf} is an a.e non-negative function with Lebesuge integral equal to one. Equalities (or inequalities) of functions will always mean equality (or inequality) almost everywhere.
\subsection{The Non-negative Lipschitz Spectrum}
In density estimation the smoothness of a pdf is often used as a measure of its complexity or how difficult the density is to estimate. The following is a characterization of a density's complexity in manner similar to the spectrum of a linear operator and is the focus of the rest of this work.
\begin{definition}\label{def:nl-spectrum}
  A pdf $p$ has a \emph{non-negative Lipschitz spectrum}, $\left(w ,L\right)$, with $w\triangleq \left(w_i \right)_{i=1}^\infty$ and $L\triangleq \left(L_i\right)_{i=1}^\infty$, if 
  \begin{equation}\label{eqn:mv-infinite}
    p = \sum_{i=1}^\infty w_i \prod_{j=1}^d p_{i,j},
  \end{equation}
  where $\left(w_i\right)_{i=1}^\infty$ is a non-negative sequence with $\sum_{i=1}^\infty w_i = 1$ and $\left(p_{i,j}\right)_{(i,j) \in \nn \times [d]}$ are Lipschitz continuous pdfs with the Lipschitz constant of $p_{i,j}$ equal to $L_i$ for all $i,j$.\footnote{As a technical convenience it will always be assumed that $w_1 >0$.} A non-negative Lipschitz spectrum will be called \emph{smooth} if the marginals $p_{i,j}$ in \eqref{eqn:mv-infinite} are all smooth.
\end{definition}
Intuitively, an NL-spectrum where $w$ decays slowly and the $L_i$'s are large indicates a more complex pdf. Later results will describe this precisely. The ``smooth'' descriptor of an NL-spectrum will not play any role for any of the results in the rest of this work. It is simply included because it is an additional regularity property that was simple to include in the proofs and may perhaps be useful in future works.

It's worth noting that the NL-spectrum of a pdf is not unique. This lack of uniqueness is not only due to trivial modifications of the spectrum, e.g. reordering or repeated summands, but may also occur since, unlike the singular value decomposition of a matrix, minimal non-negative factorizations and factorizations of tensors are not necessarily unique up to scaling and reordering. There are many works investigating the uniqueness or lack of uniqueness of (non-negative) matrix/tensor/measure factorizations \citep{kruskal77, sidiropoulos00, donoho04, symtensorrank, allman09, anandkumar14, vandermeulen15, tahmasebi18, vandermeulen19, vandermeulen22}. This lack of uniqueness might be considered a potential disadvantage of the NL-spectrum compared smoothness-based characterizations of complexity such as H\"older, Sobolev, and Nikol'ski smoothness that are characterized by one or two scalar values rather than a pair of infinite series \citep{tsybakov08}.

The following theorem shows that NL-spectra describe a very rich class of pdfs.
\begin{theorem}\label{thm:decomp}
  If a pdf is a.e. continuous then it has a smooth non-negative Lipschitz spectrum.
\end{theorem}
This set of pdfs arguably contains \emph{all} pdfs that are of practical interest; it is difficult to imagine a real-life situation where one would be interested in estimating a pdf that is essentially discontinuous on a set of positive measure. In comparison the H\"older, Sobolev, and Nikol'ski smoothness classes are not applicable to pdfs that contain a single discontinuity. 

Decompositions or approximations reminiscent of the NL-spectrum exist elsewhere in analysis and probability theory. For example, a multivariate Riemann sum has a form similar to \eqref{eqn:mv-finite},
 \begin{equation*}
   \sum_{i=1}^k w_i \ind_{\left(a_{i,1},b_{i,1}\right) \times \cdots \times \left(a_{i,d},b_{i,d}\right)} = \sum_{i=1}^k w_i \ind_{\left(a_{i,1},b_{i,1}\right)} \times \cdots \times \ind_{\left(a_{i,d},b_{i,d}\right)}= \sum_{i=1}^k w_i \prod_{j=1}^d \ind_{\left(a_{i,j},b_{i,j}\right)}.
 \end{equation*}
A similar decomposition has been mentioned in works on stochastic processes. The following quote is from \cite{kendall02}:
 \begin{displayquote}
   ...one can show existence for state-space a smooth manifold when the kernel has a continuous density $p(x, y)$, and indeed then one can show small sets of order 1 abound, in the sense that they can be used to produce a representation  
     $p(x,y) = \sum_{i=1}^\infty f_i(x)g_i(y)$, 
   where the $f_i(x)$ are non-negative continuous functions supported on small sets, and the $g_i(y)$ are probability density functions. 
 \end{displayquote}
 Theorem \ref{thm:decomp} stands apart from previous results because the NL-spectrum decomposition is exact, not an approximation, it is a mixture of pdfs, the component marginals have strong regularity properties, and it describes a very rich and practically useful class of pdfs.
\subsection{Nonparametric Estimator Results}
The following theorems show the existence of nonparametric density estimators whose performance depends on the NL-spectrum of the target density. Let $\sD_d$ be the set of all pdfs supported on the $d$-dimensional unit cube, $\left[0,1\right]^{\times d}$. The following theorem is derived using a low-rank histogram estimator introduced in \cite{vandermeulen21}.
\begin{theorem}
  \label{thm:nl-finite}
    Let $d,b,k,n \in \nn$ and $0<\delta\le 1$. There exists an estimator $V_n \in \sD_d$ such that, for any density $p\in\sD_d$, with NL-spectrum $\left(w,L\right)$, the following holds
        \begin{equation*}
          P\left(\|p-V_n\|_1 > \frac{\sqrt{3}d}{2b}\sum_{i=1}^k w_i L_i  + 6\sum_{i=k+1}^\infty w_i+ 7\sqrt{\frac{2bdk\log(4bdkn)}{n}}+7\sqrt{\frac{\log(\frac{3}{\delta})}{2n}}\right)<\delta,
        \end{equation*}
    where $V_n$ is a function of $X_1,\ldots,X_n\simiid p$.
\end{theorem}
While the estimators in this work are  technically implementable, they are computationally intractable. It will be helpful to introduce a slightly different way of characterizing NL-spectra. 
\begin{definition}
  An \emph{NL-class}, denoted by $\nlc\left(W,\tL\right)$, with $W \triangleq \left(W_i \right)_{i=1}^\infty$ a non-negative, non-increasing sequence that converges to 0, and $\tL\triangleq \left(\tL_i\right)_{i=1}^\infty$ a non-negative, non-decreasing sequence, is the set of all pdfs with an NL-spectrum $\left(w,L\right)$ satisfying $L_k \le \tL_k$ and $\sum_{i=k+1}^\infty w_i \le W_k$ for all $k$.
\end{definition}
There exist estimators with with following behavior on NL-classes.
\begin{theorem} 
  \label{thm:nlc-finite}
    Let $d,b,k,n \in \nn$ and $0<\delta\le 1$. There exists an estimator $V_n \in \sD_d$ such that, for any density $p\in\sD_d\bigcap \nlc\left(W,L\right)$, the following holds,
        \begin{equation*}
          P\left(\|p-V_n\|_1 > \frac{\sqrt{3}d}{2b} L_k  + 6W_k+ 7\sqrt{\frac{2bdk\log(4bdkn)}{n}}+7\sqrt{\frac{\log(\frac{3}{\delta})}{2n}}\right)<\delta,
        \end{equation*}
    where $V_n$ is a function of $X_1,\ldots,X_n\simiid p$.
\end{theorem}
One can apply this to get estimators that are adapted to NL-classes with polynomial rates.
\begin{proposition}\label{prop:mv-poly}
  Let $\alpha,\beta >0$ and $d \in \nn$. There exists a universally consistent sequence of estimators, $\left(V_n\right)_{n=1}^\infty$, on $\sD_d$, such that, for a fixed sampling density, $p\in\sD_d\cap \nlc(W,L)$, with $W_k \in O_k\left(k^{-\alpha}\right)$ and $L_k \in O_k\left(k^{\beta}\right)$, the following holds, $\left\|V_n - p\right\|_1 \in \tilde{O}_n\left(n^{-\alpha/\left( 3 \alpha + \beta +1 \right)} \right)$.
\end{proposition}
Universal consistency holds on $\sD_d$; if the sampling density $p\in \sD_d$ is fixed then $\left\|p-V_n\right\|_1 \cip 0$. Here the decay in $w$ works against growth in $L$ with regards to convergence rate. Leaving $\beta$ fixed and letting $\alpha \to \infty$ this estimator approaches a rate of $\tilde{O}_n\left(1/\sqrt[3]{n}\right)$ which matches the rate of convergence in \cite{vandermeulen21} for finite rank densities. This approximately matches the optimal rate of convergence for one-dimensional histograms in \cite{gyorfi85}, ``for smooth densities, the average $L^1$ error for the histogram estimate must vary at least as $n^{-1/3}$.''
\section{Proofs}\label{sec:proofs}
This section is broken up into two subsections. The fist subsection builds towards and then proves Theorem \ref{thm:decomp} and the second subsection proves estimator bounds and rates. Before proving the results of this paper, more notation must be introduced. 

Let $\sF$ be the set of smooth, Lipschitz continuous pdfs on $\rn$. Let $\sF^d \triangleq \left\{f_1 \times \cdots \times f_d \mid f_i \in \sF \right\}$, note that these are also pdfs. Recall that the \emph{conical hull} of a set $S$ in a real-valued vector space is
  $\cone\left(S\right) \triangleq \left\{ \sum_{i=1}^n w_i s_i\mid n\in \nn, w_i\ge 0,s_i\in S\right\}$.
The $\cone$ operator will only be applied to $\sF^d$.
Though standard notation, the reader is reminded that, for a set $S$, $S^{\nn}$ is the set of all infinite sequences of elements of $S$: $S^{\nn} \triangleq \left\{\left(s_i\right)_{i=1}^\infty \mid s_i \in S \right\}$. Finally $\lambda$ denotes the Lebesgue measure where dimension will always be clear from context.
\subsection{Towards a Proof of Theorem \ref{thm:decomp}}
The following lemma states that under an indicator function on a multivariate interval, $\prod_{i=1}^d\ind_{\left(a_i,b_i\right)}$, one can fit a positively scaled element of $\sF^d$ that approximates it arbitrarily well in $L^1$ distance.
\begin{lemma}\label{lem:bump-approx}
  Let $\varepsilon >0$ and $I = \prod_{i=1}^d \ind_{\left(a_{i},b_{i}\right)}$. There exists $f\in \sF^d$ and $w\ge 0$ such that $wf\le I$  and $\int I - wf d\lambda \le\varepsilon$.
\end{lemma}
\begin{proof}\textbf{of Lemma \ref{lem:bump-approx}}
  The proof will proceed by induction on dimension, $d$. \newline
  \textbf{Base Case, $d=1$:} If $\varepsilon \ge b_1-a_1$ one can simply select any $f\in \sF$ and let $w=0$, giving 
    $wf = 0 \le I$
     and 
    $\int I - wfd \lambda  = b_1 - a_1 - 0 \le \varepsilon$,
  so the lemma holds when $\varepsilon \ge b_1-a_1$. The remainder of this case will proceed with $\varepsilon < b_1-a_1$ and the ``1'' subscript dropped.
  
  The base case will be proven using smooth \emph{bump functions}; see Section 13.1 in \cite{tu10} for a technical treatment of bump functions. This proof only necessitates a very simple set of bump functions: for any $c_1<c_2<c_3<c_4\in \rn$ there exists a smooth function, $\rho$, with $\rho$ equal to zero on $\left(c_1,c_4\right)^C$, $\rho$ equal to $1$ on $\left(c_2, c_3\right)$, and $\rho$ in $[0,1]$ elsewhere.
  Using this, let $\tf$ be a smooth function like $\rho$ with $c_1 = a,  c_2 = a+\varepsilon/4, c_3 =  b -\varepsilon/4, c_4 = b$.
  Let $f = \tf/\int \tf d\lambda$ and $w = \int \tf d\lambda$ so $\tf = wf$ and $f$ is a pdf with $wf\le I$.
  
  Because $f$ is smooth its first derivative exists everywhere and is continuous. Since $df$ is continuous on the compact set $[a,b]$, $df$ bounded on $[a,b]$. Because $f$ is zero on $[a,b]^C$, $df$ is identically zero on that set.  From this it follows that $df$ is bounded and therefore $f$ is Lipschitz continuous, in addition to smooth, and thus $f\in \sF$.
  The following inequality then finishes the base case $d=1$,
  \begin{align*}
    \int& I - wf d\lambda\\
    &= \int_{(a+\varepsilon/4, b- \varepsilon/4)} \ind_{(a,b)} - wf d\lambda + \int_{(a, b)^C} \ind_{(a,b)} - wf d\lambda 
     + \int_{(a, a+\varepsilon/4)\cup (b- \varepsilon/4 ,b )} \ind_{(a,b)} - wf d\lambda\\
    &=  \int_{(a, a+\varepsilon/4)\cup (b- \varepsilon/4 ,b )} \ind_{(a,b)} - wf  d\lambda
    \le  \int_{(a, a+\varepsilon/4)\cup (b- \varepsilon/4 ,b )} 1  d\lambda
     = \varepsilon/2.
  \end{align*}
  \noindent\textbf{Induction Step:} Suppose the lemma holds for some $d \in \nn$. Let $\varepsilon >0$ and $I = \ind_{\left(a_1,b_1\right)}\times \cdots \times \ind_{\left(a_d,b_d\right)} \times \ind_{\left(a_{d+1},b_{d+1}\right)}$.
   Let $J= \ind_{\left(a_1,b_1\right)}\times \cdots \times \ind_{\left(a_d,b_d\right)}$ and $J'=\ind_{\left(a_{d+1},b_{d+1}\right)}$, so $I = J \times J'$.
   From the induction hypothesis there exist $w,w'\ge 0$, $f \in \sF^d$, and $f' \in \sF$, such that $wf \le J$, $w' f' \le J'$, 
  \begin{equation*}
    \int J - w f d\lambda \le \frac{\varepsilon}{4 \int J' d\lambda},\text{ and }
  \int J'- w' f' d \lambda \le \frac{\varepsilon}{4 \int J d \lambda   }. 
  \end{equation*}
  Note that $f\times f' \in \sF^{d+1}$ and $w w' f\times f' \le I$. The following completes the proof
  \begin{align*}
    \int J \times J' - w w' f\times f'd \lambda
    &=\int J \times J' - J \times w' f' + J \times w' f'- w w' f \times f' d \lambda\\
    &=\int J \times \left(J' - w' f'\right) + \left(J- w f\right) \times w' f'  d\lambda\\
    &=\int Jd \lambda \int \left(J' - w' f'\right) d \lambda + \int \left(J-wf\right) d \lambda \int w' f' d \lambda\\
    &\le\int J d \lambda \frac{\varepsilon}{4 \int J d\lambda} + \frac{\varepsilon}{4 \int J' d \lambda   }  \int J' d \lambda
    = \varepsilon /2.
  \end{align*}
\end{proof}
The following lemma shows that, for any function in a class of sufficiently regular non-negative functions, one can find an element in the conical hull of $\sF^d$ that fits under the function and approximates that function arbitrarily well in $L^1$ distance.

\begin{lemma}\label{lem:finite-approx}
  Let $p:\rn^d \to \rn$ be non-negative, compactly supported, bounded, and a.e. continuous and let $\varepsilon >0$. There exists $k\in \nn$, $w_1,\ldots,w_k\ge0$, and  $f_1,\ldots,f_k \in \sF^d $ such that
  \begin{equation*}
    \int p - \sum_{i=1}^k w_i f_i d\lambda \le\varepsilon \quad \text{ and } \quad p \ge \sum_{i=1}^k w_i f_i. 
  \end{equation*}
\end{lemma}

\begin{proof}\textbf{of Lemma \ref{lem:finite-approx}}
  Because $p$ is compactly supported, bounded, and a.e. continuous it is Riemann integrable as a consequence of the Riemann-Lebesgue Theorem  and therefore Darboux integrable\footnote{The facts used here are common in single variable analysis texts. For a complete treatment of the multivariate versions of these facts see Sections 2 and 3 in Chapter IV in \cite{edwards94}. Note that the terminology in that work is somewhat nonstandard.}. Thus, using a lower Darboux sum, there exists a ``step function,''
    $\phi = \sum_{i=1}^k w'_i \prod_{j=1}^d\ind_{\left(a_{i,j},b_{i,j}\right)}$,
  with $w'_i \ge 0$, such that $\phi\le p$ and $\int p-\phi d\lambda \le \varepsilon/2$. Let $I_i = \prod_{j=1}^d \ind_{\left(a_{i,j},b_{i,j}\right)}$ for brevity. From Lemma \ref{lem:bump-approx} there exists $f_1,\ldots,f_k \in \sF^d$ and $w_1,\ldots,w_k \ge 0$ such that $w_i f_i \le w'_i I_i$ and $\int w_i' I_i - w_i f_i d\lambda \le \varepsilon / (2n)$. The following now completes the proof,
  \begin{align*}
    \int p - \sum_{i=1}^k w_i f_i d\lambda 
    &= \int p - \sum_{i=1}^k w'_i I_i + \sum_{i=1}^k w'_i I_i - \sum_{i=1}^k w_i f_i d\lambda\\
    &= \left(\int p - \sum_{i=1}^k w'_i I_i d \lambda \right) +\left(\sum_{i=1}^k\int w'_i I_i -  w_i f_i d\lambda\right)\\
    &\le \varepsilon/2 +  n \frac{\varepsilon}{2n} = \varepsilon.
  \end{align*}
\end{proof}
\begin{lemma}\label{lem:compact-decomposition}
  Let $p:\rn^d \to \rn$ be non-negative, compactly supported, bounded, and a.e. continuous. There exists a non-negative sequence $\left(w_i\right)_{i=1}^\infty$ and $\left(f_i \right)_{i=1}^\infty \in \left(\sF^d\right)^{\nn} $ such that $p = \sum_{i=1}^\infty w_i f_i$.
\end{lemma}

\begin{proof}\textbf{of Lemma \ref{lem:compact-decomposition}}
  From Lemma \ref{lem:finite-approx} there exists $q_1 \in \cone \left(\sF^d\right)$ such that $p\ge q_1$ and $\int p - q_1 d\lambda \le1/2$. Note that $p-q_1$ is non-negative, compactly supported, bounded, and a.e continuous. Proceeding by induction there exists a sequence $\left(q_i\right)_{i=1}^\infty \in \cone\left(\sF^d\right)^{\nn}$ satisfying $\int p -\sum_{i=1}^k q_i d\lambda \le1/2^k$ and $p \ge \sum_{i=1}^k q_i$, for all $k$. Since, $\int p -\sum_{i=1}^k q_i d\lambda =\int \left|p -\sum_{i=1}^k q_i \right| d\lambda  = \left\| p -\sum_{i=1}^k q_i\right\|_1$, the summation is a Cauchy sequence in $L^1$ and $\sum_{i=1}^\infty q_i = p$. Since $q_i \in \cone\left(\sF^d \right)$ one can write $q_i = \sum_{j=1}^{k_i}w_{i,j} f_{i,j}$, with $w_{i,j}\ge 0, f_{i,j} \in \sF^d$.
  From this,
    $p = \sum_{i=1}^\infty \sum_{j=1}^{k_i}w_{i,j} f_{i,j}$.
  Observe that, 
     $\left\|p\right\|_1 = \left\|\sum_{i=1}^\infty \sum_{j=1}^{k_i}w_{i,j} f_{i,j}\right\|_1 < \infty$,
  and since $w_{i,j} f_{i,j}$ are non-negative functions it follows that
  \begin{equation}
    \infty 
    > \left\|\sum_{i=1}^\infty \sum_{j=1}^{k_i}w_{i,j} f_{i,j}\right\|_1 
    = \sum_{i=1}^\infty \left\|\sum_{j=1}^{k_i}w_{i,j} f_{i,j}\right\|_1 
    = \sum_{i=1}^\infty \sum_{j=1}^{k_i}\left\|w_{i,j} f_{i,j}\right\|_1. \label{eqn:nonneg-summable}
  \end{equation}
  So the set of all $w_{i,j}f_{i,j}$ is absolutely summable in $L^1$.  Thus, by relabeling, there exist $w_i \ge 0$ and $f_i \in \sF^d$ such that  $\sum_{i=1}^\infty w_i f_i = p$.
\end{proof}
Lemma \ref{lem:compact-decomposition} can now be generalized to prove Theorem \ref{thm:decomp}.
\begin{proof}\textbf{of Theorem \ref{thm:decomp}}
  Let $p$ be an a.e. continuous pdf. One can clearly decompose $p$ into the absolutely summable summation $p=\sum_{i=1}^\infty p_i$ where $p_i$ are all non-negative, a.e. continuous, compactly supported functions. For example, one can simply break $p$ into components supported on $\left\{\prod_{i=1}^d \left[a_i, a_i+1 \right]\mid a_1,\ldots,a_d \in \mathbb{Z} \right\}$ which is a countable set. One summand in this summation will be considered, so fix some $p_i$. 
  
  It will now be shown that $p_i = \sum_{j=0}^\infty q_{i,j}$ where $q_{i,j}$ are bounded, a.e. continuous, compactly supported, non-negative functions. For this define $q_{i,j} = \max\left(\min\left(p_i-j,1\right),0 \right)$.
  Clearly $1\ge q_{i,j} \ge 0$ for all $j$ so $q_{i,j}$ are bounded and non-negative. For arbitrary $x$, if $p_i(x) = 0$ then 
  \begin{equation*}
    q_{i,j}(x) 
    = \max\left(\min\left(p_i(x)-j,1\right),0 \right)
    = \max\left(\min\left(0-j,1\right),0 \right)
    = \max\left(0-j,0 \right)
    = 0,
  \end{equation*} 
  so $p_i(x) = 0 \Rightarrow q_{i,j}(x)=0$, and thus $\supp\left(q_{i,j}
  \right) \subseteq \supp\left(p_i\right)$ and since the support of $p_i$ is compact the support of $q_{i,j}$ is compact. The $q_{i,j}$ are a.e continuous by virtue of the fact that $\max$, $\min$, and subtraction preserve pointwise continuity.

  It will now be shown that $\sum_{j=0}^\infty q_{i,j} = p_i$. Consider some arbitrary $x\in \rn^d$. There exists some $\ell\in \nn\cup\{0\}$ such that $p_i(x) \in [\ell,\ell+1]$. It follows that
  \begin{align*}
    \sum_{j=0}^\infty q_{i,j}(x)
    =&\sum_{j=0}^{\ell-1} \max\left(\min\left(p_i(x)-j,1\right),0 \right)
    + \max\left(\min\left(p_i(x)-\ell,1\right),0 \right)\\
    &\cdots +\sum_{j=\ell+1}^{\infty}\max\left(\min\left(p_i(x)-j,1\right),0 \right)\\
    =&\sum_{j=0}^{\ell-1} \max\left(1,0 \right)
    + \max\left(p_i(x)-\ell,0 \right)
    +\sum_{j=\ell+1}^{\infty}\max\left(p_i(x)-j,0 \right)\\
    =&\ell+ p_i(x)-\ell+\sum_{j=\ell+1}^{\infty}0
    =p_i(x).
  \end{align*}
    Since the $q_{i,j}$ are all non-negative, using the same argument as \eqref{eqn:nonneg-summable} in the proof of Lemma \ref{lem:compact-decomposition}, it follows that the sum $\sum_{i=1}^\infty \sum_{j=0}^ \infty q_{i,j}$ is $L^1$ absolutely summable.
    From Lemma \ref{lem:compact-decomposition} it follows that for all $q_{i,j}= \sum_{\ell=1}^\infty w_{i,j,\ell} f_{i,j,\ell}$ with $w_{i,j,\ell} \ge 0$ and $f_{i,j,\ell} \in \sF^d$ and $p$. Substituting these in yields $p = \sum_{i=1}^\infty  \sum_{j=0}^\infty\sum_{\ell=1}^\infty w_{i,j,\ell} f_{i,j,\ell} =p$, which again is absolutely summable. Because this summation is over a countable set, one can relabel $p = \sum_{i=1}^\infty w_i f_i$. Because $f_i \in \sF^d$, $f_i = \prod_{j=1}^d f_{i,j}$ with $f_{i,j}$ being Lipschitz continuous, a sequence can be constructed, $\left(L_i\right)_{i=1}^\infty$, such that $f_{i,j}$ are all $L_i$-Lipschitz continuous. Finally note $1 = \int \sum_{i=1}^\infty w_i f_i =  \sum_{i=1}^\infty w_i$. It follows that $p$ has NL-spectrum $(w,L)$.
\end{proof}

\subsection{Proofs of Estimator Bounds and Rates}
The estimator results build upon results in \cite{vandermeulen21} that analyzed histogram-style estimators of densities on $d$-dimensional unit cubes. Before presenting these results a bit more notation is required. In \cite{vandermeulen21} the authors define a notion of a low-rank histogram. For this context a ``histogram'' refers to a pdf that is constant on a partition containing equally spaced cubes. In particular they define $\sH_{1,b}$ to be the set of all histograms on the unit interval with $b$ evenly spaced bins i.e.  
\begin{equation*}
  \sH_{1,b} \triangleq \left\{\sum_{i=1}^b w_i \ind_{[(i-1)/b,i/b]}\mid w_i\ge 0, \int \sum_{i=1}^b w_i \ind_{[(i-1)/b,i/b]}d\lambda = 1  \right\}.
\end{equation*}
A low-rank histogram on $d$-dimensional space, with $b$ bins per dimension, is then defined as
\begin{equation*}
  \sH_{d,b}^k \triangleq \left\{\sum_{i=1}^k w_i \prod_{j=1}^d h_{i,j}\mid 0 \le w_i, \sum_{i=1}^k w_i = 1, h_{i,j} \in \sH_{1,b} \right\}.
\end{equation*}
Let $\lip_L$ be the set of $L$-Lipschitz continuous from $\rn$ to $\rn$.

The following proposition was proven in \cite{vandermeulen21} and gives a bound on how well one can select estimators from $\sH_{d,b}^k$.
\begin{proposition}[Proposition 2.1 from \cite{vandermeulen21}]
  \label{prop:finiteant}
    Let $d,b,k,n \in \nn$ and $0<\delta\le 1$. There exists an estimator $V_n \in \sH_{d,b}^k$ such that 
        \begin{equation}
          \sup_{p \in \sD_d}P\left(\|p-V_n\|_1 > 3\min_{q\in\sH_{d,b}^k}\|p-q\|_1+ 7\sqrt{\frac{2bdk\log(4bdkn)}{n}}+7\sqrt{\frac{\log(\frac{3}{\delta})}{2n}}\right)<\delta,\label{eqn:finiteant}
        \end{equation}
    where $V_n$ is a function of $X_1,\ldots,X_n\simiid p$.
\end{proposition}
To begin proving Theorem \ref{thm:nl-finite}, the bias term, $\min_{q\in\sH_{d,b}^k}\|p-q\|_1$, is analyzed via NL-spectrum.

The following is a bound on rank-one bias that simplifies the analysis in \cite{vandermeulen21} (c.f. Theorem 2.5 in that paper).
\begin{lemma}\label{lem:l1approx}
  Let $f_1,\ldots,f_d \in \lip_L\cap \sD_1$. Then
    $\min_{h \in \sH_{d,b}^1} \left\|\prod_{i=1}^d f_i - h \right\|_1 \le \frac{dL}{\sqrt{12}b}$.
\end{lemma}
  Note that a minimizer exists due to the compactness of the set $\sH_{d,b}^1$ (see Appendix A.1 in \cite{vandermeulen21}).
\begin{proof}\textbf{of Lemma \ref{lem:l1approx}}
  To begin,
  \begin{align}
    \min_{h \in \sH_{d,b}^1}&\left\|\prod_{i=1}^d f_i - h \right\|_1\notag\\
    & = \min_{h_1,\ldots,h_d \in \sH_{1,b}} \left\|\prod_{i=1}^d f_i - \prod_{j=1}^d h_j \right\|_1
     \le \min_{h_1,\ldots,h_d \in \sH_{1,b}} \sum_{i=1}^d\left\| f_i -  h_i \right\|_1 & \text{}\label{eqn:reiss}\\
    & = \min_{h_1,\ldots,h_d \in \sH_{1,b}} \sum_{i=1}^d\left\| \ind_{[0,1]} \cdot \left(f_i -  h_i \right) \right\|_1
     \le \min_{h_1,\ldots,h_d \in \sH_{1,b}} \sum_{i=1}^d\left\| \ind_{[0,1]}\right\|_2 \left\| f_i -  h_i  \right\|_2 \label{eqn:holder} \\
    & \le \min_{h_1,\ldots,h_d \in \sH_{1,b}} \sum_{i=1}^d \frac{L}{\sqrt{12}b}  \label{eqn:explain} 
    = \frac{dL}{\sqrt{12}b},
  \end{align}
  where \eqref{eqn:reiss} is a consequence of Lemma 3.3.7 in \cite{reiss89} (see Appendix \ref{appx:additional}) and \eqref{eqn:holder} follows from H\"older's inequality.
  To see \eqref{eqn:explain} note that Lemma C.2 and C.4 in \cite{vandermeulen21} state that, for $f_i \in \sD_1\cap \lip_L$,
  \begin{equation*}
    \left\|f_i - \proj_{\spn\left(\sH_{1,b}\right)}f_i\right\|_2 \le \frac{L}{\sqrt{12}b} \text{ and }
    \left\|f_i - \proj_{\spn\left(\sH_{1,b}\right)}f_i\right\|_2 = \min_{h\in \sH_{1,b}} \left\|f_i - h\right\|_2,
  \end{equation*}
  with the projection being in $L^2$ distance. This finishes the proof.
\end{proof}
In \cite{vandermeulen21} the authors investigated pdfs whose domains were $[0,1]^{\times d}$, not $\rn^d$. This is a rather subtle point, but this means that the density $p(x) =2x\ind_{[0,1]}(x)$ would be a valid 2-Lipschitz continuous density in that paper, but it is not in this work. The results from that work are applicable to densities in $\sD_d\cap \lip_L$. The results in this work could be adjusted to be slightly more general on the domain $[0,1]^{\times d}$, but this was omitted for simplicity.

The next lemma characterizes the bias in terms of NL-spectrum.
\begin{lemma} \label{lem:nl-approx}
  Let $p \in \sD_d$ have NL-spectrum $(w,L)$, then
  \begin{equation*}
    \min_{\sH_{d,b}^k} \left\|p - h \right\|_1 \le\frac{d}{\sqrt{12}b}\sum_{i=1}^k w_i L_i  + 2\sum_{i=k+1}^\infty w_i. 
  \end{equation*} 
\end{lemma}

\begin{proof}\textbf{of Lemma \ref{lem:nl-approx}}
  Because $p$ has NL-spectrum $(w,L)$, $p = \sum_{i=1}^\infty w_i \prod_{j=1}^d p_{i,j}$,
  with $p_{i,j} \in \sD_d \cap \lip_{L_i}$ and $w_i\ge0$ for all $i$.
  Let $\tp_i = \arg \min_{h \in \sH_{d,b}^1}\left\|\prod_{j=1}^d p_{i,j} - h\right\|_1$ and $\tw_i = w_i/\left(\sum_{j=1}^kw_j \right)$ ($\tp_i$ exists since Lemma \ref{lem:l1approx} is a minimum and not an infimum recall that $w_1>0$ in Definition \ref{def:nl-spectrum} so $\tw_i$ all exist). So $\sum_{i=1}^k \tw_i \tp_i \in \sH_{d,b}^k$. Now the following bound can be shown,
  \begin{align}
    \left\|\sum_{i=1}^\infty w_i p_i - \sum_{j=1}^k \tw_j\tp_j\right\|_1
    &\le \left\|\sum_{i=1}^k w_i p_i - \sum_{j=1}^k \tw_j\tp_j\right\|_1 + \left\| \sum_{i=k+1}^\infty w_i p_i\right\|_1 \notag \\
    &= \left\|\sum_{i=1}^k w_i p_i - w_i \tp_i +w_i \tp_i -  \tw_i\tp_i\right\|_1 
    +  \sum_{i=k+1}^\infty  w_i \notag\\
    &\le \left\|\sum_{i=1}^k w_i p_i - w_i \tp_i\right\|_1
    + \left\|\sum_{i=1}^kw_i \tp_i -  \tw_i\tp_i\right\|_1 
    +  \sum_{i=k+1}^\infty  w_i. \label{eqn:three-terms}
  \end{align}
  Using Lemma \ref{lem:l1approx}, the following bounds the left term of \eqref{eqn:three-terms}
  \begin{align*}
    \left\|\sum_{i=1}^k w_i p_i - w_i \tp_i\right\|_1
    \le \sum_{i=1}^k w_i\left\|  p_i - \tp_i\right\|_1
    \le \sum_{i=1}^k w_i \frac{dL_i}{\sqrt{12}b}
    = \frac{d}{\sqrt{12}b} \sum_{i=1}^k w_i L_i.
  \end{align*}
  The following bounds the center term of \eqref{eqn:three-terms}, finishing the proof,
  \begin{align}
    \left\|\sum_{i=1}^kw_i \tp_i -  \tw_i\tp_i\right\|_1 
    &\le \sum_{i=1}^k\left|w_i-\tw_i\right| \left\| \tp_i  \right\|_1 
    = \sum_{i=1}^k\left|w_i-w_i/\left(\sum_{j=1}^k w_j \right)\right|\label{eqn:nl-approx-ref}\\
    &= \left(\sum_{i=1}^kw_i \right)\left|1-\left(\sum_{j=1}^k w_j \right)^{-1}\right|  
    = \left|\left(\sum_{i=1}^kw_i \right)-1\right|  
    = \sum_{i=k+1}^\infty  w_i.\notag
  \end{align}
\end{proof}

\begin{proof}\textbf{of Theorem \ref{thm:nl-finite}}
  From Proposition \ref{prop:finiteant}, for all $d,b,k,n \in \nn$ and $0<\delta\le 1$. There exists an estimator $V_n \in \sH_{d,b}^k$ such that 
        \begin{equation*}
          \sup_{\tp \in \sD_d}P\left(\|\tp-V_n\|_1 > 3\min_{q\in\sH_{d,b}^k}\|\tp-q\|_1+ 7\sqrt{\frac{2bdk\log(4bdkn)}{n}}+7\sqrt{\frac{\log(\frac{3}{\delta})}{2n}}\right)<\delta,
        \end{equation*}
    where $V_n$ is a function of $X_1,\ldots,X_n\simiid \tp$. For this $V_n$ it thus follows that
        \begin{equation}\label{eqn:nl-finite}
          P\left(\|p-V_n\|_1 > 3\min_{q\in\sH_{d,b}^k}\|p-q\|_1+ 7\sqrt{\frac{2bdk\log(4bdkn)}{n}}+7\sqrt{\frac{\log(\frac{3}{\delta})}{2n}}\right)<\delta. 
        \end{equation}

  From Lemma \ref{lem:nl-approx},
  $\min_{\sH_{d,b}^k} \left\|p - h \right\|_1 \le\frac{d}{\sqrt{12}b}\sum_{i=1}^k w_i L_i  + 2\sum_{i=k+1}^\infty w_i$. So
  \begin{equation*}
    3\min_{q\in\sH_{d,b}^k}\|p-q\|_1 
    \le 3\left(\frac{d}{\sqrt{12}b}\sum_{i=1}^k w_i L_i  + 2\sum_{i=k+1}^\infty w_i\right)
    \le \frac{\sqrt{3}d}{2b}\sum_{i=1}^k w_i L_i  + 6\sum_{i=k+1}^\infty w_i.
  \end{equation*}
  The theorem follows from substituting this back into \eqref{eqn:nl-finite}.
\end{proof}

\begin{lemma} \label{lem:nlc-approx}
  Let $p \in \sD_d\cap \nlc(W,\tL)$, then
    $\min_{\sH_{d,b}^k} \left\|p - h \right\|_1 \le\frac{d}{\sqrt{12}b}L_k  + 2W_k$. 
\end{lemma}
\begin{proof}\textbf{of Lemma \ref{lem:nlc-approx}}
  Because $p\in \nlc(W,\tL)$, $p$ can be written as, $p = \sum_{i=1}^\infty w_i \prod_{j=1}^d p_{i,j}$,
  with, $w_i \ge 0$, $\sum_{i=k+1}^\infty w_i \le W_k$ for all $k$, and $p_{i,j} \in \sD_d \cap \lip_{L_i}$, with $L_i \le \tL_k$ for all $i\le k$. Using this NL-spectrum the proof proceeds exactly as the proof of Lemma \ref{lem:nl-approx} to \eqref{eqn:nl-approx-ref}, yielding
  \begin{equation*}
    \min_{h \in \sH_{d,b}^k}\left\|p - h\right\|_1\le \sum_{i=1}^k w_i \frac{dL_i}{\sqrt{12}b}
    +\sum_{i=1}^k\left|w_i-w_i/\left(\sum_{j=1}^k w_j \right)\right|  
    +  \sum_{i=k+1}^\infty  w_i.
  \end{equation*}
  Since $\tL$ is non-decreasing and $w_i \ge 0$ with $\sum_{i=1}^k w_i \le 1$ it follows that 
\begin{align*}
  \sum_{i=1}^k w_i \frac{dL_i}{\sqrt{12}b}
  &\le \sum_{i=1}^k w_i \frac{d\tL_i}{\sqrt{12}b}
  \le \sum_{i=1}^k w_i \frac{d\tL_k}{\sqrt{12}b}
  \le \frac{d\tL_k}{\sqrt{12}b}.
\end{align*}
  Using the same argument as in the proof of Lemma \ref{lem:nl-approx} it follows that
  \begin{align*}
    \sum_{i=1}^k\left|w_i-w_i/\left(\sum_{j=1}^k w_j \right)\right|  + \sum_{i=k+1}^\infty  w_i
    & = 2\sum_{i=k+1}^\infty  w_i \le 2W_k.
  \end{align*}
\end{proof}

\begin{proof}\textbf{of Theorem \ref{thm:nlc-finite}}
  Proceed as in the proof of Theorem \ref{thm:nl-finite}, with Lemma \ref{lem:nlc-approx} for Lemma \ref{lem:nl-approx}.
\end{proof}

\begin{proof}\textbf{of Proposition \ref{prop:mv-poly}}
  Let $V_n$ be the estimator from Theorem \ref{thm:nlc-finite} with $k = \left \lceil n^{1/(3 \alpha + \beta +1)} \right \rceil$ and $b = \left\lceil n^{(\alpha + \beta)/(3\alpha + \beta + 1)}\right\rceil$. Note that \eqref{eqn:finiteant} also holds for $V_n$. Let $p\in\sD_d\cap \nlc(W,L)$ with $W_k \in O\left(k^{-\alpha}\right)$ and $L_k \in O\left(k^{\beta}\right)$. From Theorem \ref{thm:nlc-finite} there is the following bound
   \begin{equation*}
          P\left(\|p-V_n\|_1 > \frac{\sqrt{3}d}{2b} L_k  + 6W_k+ 7\sqrt{\frac{2bdk\log(4bdkn)}{n}}+7\sqrt{\frac{\log(\frac{3}{\delta})}{2n}}\right)<\delta. 
        \end{equation*}
  It will be shown that the four terms right of the inequality go to zero at rate $\tilde{O}\left(n^{-\alpha/\left( 3 \alpha + \beta +1 \right)} \right)$. First observe that 
    $k\in \Theta_n\left(   n^{1/(3 \alpha + \beta +1)}\right)$
    and
    $b\in \Theta_n\left( n^{(\alpha + \beta)/(3\alpha + \beta + 1)}\right)$.
  For the first two terms, for $n$ sufficiently large,
  \begin{align*}
    \frac{\sqrt{3}d}{2b}L_k + 6W_k
    &\le C \frac{k^{\beta}}{n^{(\alpha + \beta)/(3\alpha + \beta + 1)}} + C k^{-\alpha} & \text{($C$ chosen sufficiently large)} \\
    &\le C' \frac{n^{\beta/(3 \alpha + \beta +1)}}{n^{(\alpha + \beta)/(3\alpha + \beta + 1)}} + C' n^{-\alpha/(3\alpha +\beta +1)} & \text{(substitute in $n$; $C'$ suf. large)}  \\
    &=  C' n^{-\alpha/(3\alpha +\beta +1)} + C' n^{-\alpha/(3\alpha +\beta +1)}.
  \end{align*}
  For sufficiently large $C$ The third term can bounded as
    $7\sqrt{\frac{2bdk\log(4bdkn)}{n}}
    \le C\sqrt{\frac{bk}{n}} \sqrt{\log(4bdkn)}$.
  For the first term, for sufficiently large $n$ and for $C$ chosen sufficiently large
  \begin{align*}
    \sqrt{\frac{bk}{n}}
    &\le  C\sqrt{\frac{n^{(\alpha + \beta)/(3\alpha + \beta + 1)}n^{1/(3 \alpha + \beta +1)}}{n}}
    =  C\sqrt{\frac{n^{(\alpha + \beta)/(3\alpha + \beta + 1)}n^{1/(3 \alpha + \beta +1)}}{n^{(3\alpha + \beta + 1)/(3\alpha + \beta + 1)}}}\\
    &= \sqrt{n^{-2\alpha/(3\alpha + \beta +1)}}
    = n^{-\alpha/(3\alpha + \beta +1)}.
  \end{align*}
  Since $4bdkn \in O_n\left(\poly(n)\right)$, it follows that $\sqrt{\log(4bdkn)} \in O_n \left(\sqrt{\log(n)}\right)$. Combining the previous two rates yields
  \begin{equation*}
    7\sqrt{\frac{2bdk\log(4bdkn)}{n}} \in O_n\left( n^{-\alpha/(3\alpha + \beta +1)}\sqrt{\log(n)} \right) \subset \tilde{O}_n\left( n^{-\alpha/(3\alpha + \beta +1)} \right).
   \end{equation*}
   Finally $7\sqrt{\frac{\log(\frac{3}{\delta})}{2n}}\in O_n\left(1/\sqrt{n}\right)$. Since $\alpha,\beta >0$ it follows that $\alpha/(3\alpha + \beta +1) < 1/3$ and
   \begin{align*}
     7\sqrt{\frac{\log(\frac{3}{\delta})}{2n}}\in O_n\left(n^{-1/2}\right) \subset O_n\left( n^{-\alpha/(3\alpha + \beta +1)} \right),
   \end{align*}
   which finishes the rate portion of this proof.
   For universal consistency \eqref{eqn:finiteant} will be used. It has already been shown that the right two summands in \eqref{eqn:finiteant} go to zero. The following lemma is also from \cite{vandermeulen21} and demonstrates that the bias term goes to zero for all densities in $\sD_d$, finishing the proof.
\begin{lemma}[Lemma 2.1 from \cite{vandermeulen21}]\label{lem:lrbias}
  Let $p \in \sD_d$. If $k\to \infty$ and $b \to \infty$ then $\min_{q\in \sH_{d,b}^k} \left\|p - q\right\|_1 \to 0$.
\end{lemma}
\end{proof}
\acks{This work was supported by the Federal Ministry of Education and Research (BMBF) for the Berlin Institute for the Foundations of Learning and Data (BIFOLD) (01IS18037A).}

\bibliography{blah.bib}

\appendix
\section{Notation} \label{appx:notation}
\begin{tabular}{|l|l|}
  \hline
  $A\times B$ & Cartesian Product
  \\\hline
  $f\times g$ & $(f\times g)(a,b) = f(a)g(b)$
  \\\hline
  $f\cdot g$ & $(f\cdot g)(a) = f(a)g(a)$
  \\\hline
  pdf& a non-negative function that integrates to one
  \\\hline
  $\nn$& $\left\{1,2,3,\ldots\right\}$
  \\\hline
  $\ind_A$ & indicator function on $A$
  \\\hline
  $\left[n\right]$& $\left\{1,2,\ldots,n\right\}$
  \\\hline
  $\sD_d$& set of all pdfs whose support is contained in $[0,1]^{\times d}$
  \\\hline
  $\sF$ & the set of all smooth, Lipschitz continuous pdfs on $\rn$.
  \\\hline
  $\sF^d$ &  $\left\{f_1 \times \cdots \times f_d \mid f_i \in \sF \right\}$
  \\\hline
  $\cone\left(S\right)$& $\left\{ \sum_{i=1}^n w_i s_i\mid n\in \nn, w_i\ge 0,s_i\in S\right\}$
  \\\hline
  $S^\nn$ & the set of all infinite sequences of $S$, i.e. $\left\{\left(s_i\right)_{i=1}^\infty \mid s_i \in S\right\}$
  \\\hline
  $\lambda$ & the Lebesgue measure, dimension is left implicit
  \\\hline
  $\lip_L$ & the set of real valued functions with domain $[0,1]$ that are $L$-Lipschitz continuous
  \\\hline
  $\sH_{1,b}$& $\left\{\sum_{i=1}^b w_i \ind_{[(i-1)/b,i/b)}\mid w_i\ge 0, \int \sum_{i=1}^b w_i \ind_{[(i-1)/b,i/b)}d\lambda = 1  \right\}$
  \\\hline
  $\sH_{d,b}^k$ &  $\left\{\sum_{i=1}^k w_i \prod_{j=1}^d h_{i,j}\mid w_i \in \Delta_k, h_{i,j} \in \sH_{1,b} \right\}$
  \\\hline
  \end{tabular}
\section{External Results}\label{appx:additional}
\begin{lemma}[Lemma 3.3.7 in \cite{reiss89}]
For probability measures $\mu_i,\nu_i$, $i=1,\ldots,d$,
  \begin{equation*}
    \left\|\prod_{i=1}^d \mu_i - \prod_{i=1}^d \nu_i\right\| \le \sum_{i=1}^d \left\|\mu_i - \nu_i\right\|,
\end{equation*}
  where the norm is the total variation norm. (Note that the total variation norm is equivalent to $L^1$ for pdfs.)
\end{lemma}

\end{document}